\documentclass{icm2010}
\usepackage{amssymb}
\usepackage{hyperref}
\usepackage{mathrsfs}
\usepackage{multirow}
\usepackage{setspace}

\title[Homotopical group theory]{The classification of $p$--compact groups and homotopical group theory}
\author[J. Grodal]{Jesper Grodal
\thanks{Supported by an ESF EURYI award and the Danish National Research Foundation.}
}
\contact[jg@math.ku.dk]{Department of Mathematical Sciences, University of
  Copenhagen\\ Universitetsparken 5, DK-2100 Copenhagen, Denmark}

\newcounter{thecounter}
\numberwithin{thecounter}{section}

\newtheorem{prop}[thecounter]{Proposition}
\newtheorem{thm}[thecounter]{Theorem}

\newtheorem{defn}[thecounter]{Definition}

\theoremstyle{definition}

\newtheorem{rem}[thecounter]{Remark}

\numberwithin{equation}{section}

\newcommand{\map}{\operatorname{map}}

\newcommand{\ho}[1]{\operatorname{Ho}(#1)}

\newcommand{\ch}{\operatorname{char}}

\newcommand{\Aut}{\operatorname{Aut}}
\newcommand{\Rep}{\operatorname{Rep}}

\newcommand{\Out}{\operatorname{Out}}

\newcommand{\Hom}{\operatorname{Hom}}
\newcommand{\im}{\operatorname{im}}

\newcommand{\hocolim}{\operatorname{hocolim}}

\newcommand{\cC}{{\mathcal{C}}}
\newcommand{\calD}{{\mathcal{D}}}

\newcommand{\bO}{{\mathbf O}}

\newcommand{\op}{\operatorname{op}}

\newcommand{\Spaces}{\operatorname{{\bf Spaces}}}
\newcommand{\pcom}{\hat{{}_p}}

\newcommand{\twocom}{\hat{{}_2}}

\newcommand{\GL}{\operatorname{GL}}
\newcommand{\SO}{\operatorname{SO}}
\newcommand{\SU}{\operatorname{SU}}
\newcommand{\Sp}{\operatorname{Sp}}
\newcommand{\DI}{\operatorname{DI}}

\newcommand{\PU}{\operatorname{PU}}

\newcommand{\Spin}{\operatorname{Spin}}

\newcommand{\Q}{{\mathbb {Q}}}

\newcommand{\F}{{\mathbb {F}}}
\newcommand{\Z}{{\mathbb {Z}}}

\newcommand{\D}{{\mathbf D}}
\newcommand{\CP}{{\mathbb{C}P}}
\newcommand{\RP}{{\mathbb{R}P}}

\newcommand{\HP}{{\mathbb{H}P}}

\newcommand{\bA}{{\mathbb{A}}}
\newcommand{\A}{{\mathbf A}}

\newcommand{\W}{{\mathcal{W}}}
\newcommand{\N}{{\mathcal{N}}}
\newcommand{\C}{{\mathbf{C}}}
\newcommand{\cZ}{{\mathcal{Z}}}
\newcommand{\bbC}{{\mathbb{C}}}

\newcommand{\beq}{\begin{eqnarray*}}
\newcommand{\eeq}{\end{eqnarray*}}

\newcommand{\tuborg}{\left\{\begin{array}{ll}}
\newcommand{\sluttuborg}{\end{array}\right.}

\newfont{\bm}{msbm10}

\newcommand{\U}{\mathrm{U}}

\newcommand{\Sol}{\operatorname{Sol}}
\newcommand{\Gr}{\operatorname{Gr}}
\newcommand{\rank}{\operatorname{rank}}
\input xypic
\xyoption{all}
\CompileMatrices

\newcommand{\dZ}{\breve \cZ}
\newcommand{\dT}{\breve T}

\renewcommand{\phi}{\varphi}
\newcommand*\InsertTheoremBreak{
  \begingroup 
    \setlength\itemsep{0pt}
    \setlength\parsep{0pt}
    \item[\vbox{\null}]
  \endgroup
}
\def\co{\colon\thinspace}

\begin{document}
\begin{abstract} 
We survey some recent advances in the homotopy theory of classifying
spaces, and homotopical group theory. We focus on the classification
of $p$--compact groups in terms of root data over the $p$--adic
integers, and discuss some of its consequences e.g.\ for finite loop spaces and polynomial
cohomology rings.
\end{abstract}



\maketitle

\bigskip

Groups are ubiquitous in real life, as symmetries of geometric
objects. For many purposes in mathematics, for instance in bundle theory, it is however not the group itself but rather its
classifying space, which takes center stage. 
The classifying space encodes the group multiplication directly in a topological space,
to be studied and manipulated using the toolbox of homotopy theory. This leads
to the idea of homotopical group theory, that one should try to do group theory in terms of classifying spaces.

The idea that there should be a homotopical version of group theory is an old one. 
The seeds were sown already in the 40s and 50s with the work of
Hopf and Serre on finite $H$--spaces and loop
spaces, and these objects were intensely studied in the 60s using the techniques of Hopf
algebras, Steenrod operations, etc., in the hands of Browder,
Thomas, and others. A bibliography containing
347 items was collected by James in 1970
\cite{james71}; see also \cite{kane88} for a continuation.

In the same year, Sullivan, in his widely circulated MIT notes
\cite{sullivan05,sullivan74}, laid out a theory of $p$--completions of topological
spaces, which had a profound influence on the subject. On the one hand it
provided an infusion of new exotic examples, laying old hopes and conjectures
to rest. On the other hand his theory of $p$--completions seemed to indicate
that the dream of doing group theory on the level of classifying spaces
could still be valid, if one is willing to replace real life, at
least temporarily, by a $p$--adic existence. However, the tools for seriously
digging into the world of $p$--complete
spaces were at the time insufficient, a stumbling block being the so-called
Sullivan conjecture \cite[p.~179]{sullivan05} relating fixed-points to homotopy fixed-points, at a
prime $p$.

The impasse ended with the solution of the Sullivan conjecture
by Miller \cite{miller84short}, and the work of Carlsson \cite{carlsson84}, reported on at this congress
in 1986
\cite{miller87short,carlsson87short}, followed by the development of
``Lannes theory'' \cite{lannes92short,lannes95icm} giving effective tools for calculating
homotopy fixed-points and maps between classifying spaces.
This led to a spate of progress. Dwyer and Wilkerson \cite{DW94} defined the notion of a
$p$--compact group, a $p$--complete version of a finite loop space, and showed
that these objects posses much of the structure
of compact Lie groups:
maximal tori, Weyl groups, etc.
In parallel to this, Jackowski, McClure, and Oliver \cite{JMO92} combined Lannes theory with space-level
decomposition techniques and sophisticated homological algebra
calculations to get precise
information about maps between classifying spaces of compact Lie groups,
that used to be out of reach.
These developments were described at this congress
in 1998 \cite{dwyer98short,oliver98short}.

The aim here is to report on some recent progress, building on the above
mentioned achievements. In particular, a complete classification of
$p$--compact groups has recently been obtained in collaborations involving
the author
\cite{AGMV08,AG09}. It
states that connected $p$--compact groups are classified by their root data
over the $p$--adic integers $\Z_p$
(once defined!), completely analogously to the
classification of compact connected Lie groups by root data over $\Z$.
It has in turn allowed for the solution of a number
of problems and conjectures dating from the
60s and 70s, such as the Steenrod problem of realizing polynomial cohomology
rings and the so-called maximal torus conjecture giving a completely
homotopical description of compact Lie groups. By local-to-global
principles the classification of $p$--compact groups furthermore provides a
quite complete understanding of what finite loop spaces look like,
integrally as well as rationally.

Homotopical group theory has branched out considerably over the last decade. There is
now an expanding theory of homotopical versions of finite
groups, the so-called $p$--local finite groups, showing signs of
strong connections to deep questions in finite group theory, such as
the classification of finite simple groups. There has been progress on homotopical group
actions, providing in some sense a homotopical version of the ``geometric
representation theory'' of tom Dieck
\cite{tomdieck87short}. And there is even evidence that certain aspects of the theory
might extend to Kac--Moody groups and other classes of groups. We shall
only be able to provide very small appetizers to some of these last
developments, but we hope
that they collectively serve as an inspiration to the reader to try to
take a more homotopical approach to his or her favorite class of groups.

\smallskip

This paper is structured as follows: Section~\ref{section:root} is an
algebraic prelude, discussing the theory of $\Z_p$--root data---the impatient reader
can skip it at first, referring back to it as needed.
Section~\ref{section:pcg} gives the definition and basic properties of
$p$--compact groups, states the classification theorem, and outlines
its proof. It also
presents various structural consequences for $p$--compact groups.
Section~\ref{section:loop} discusses applications to finite loop
spaces such as an algebraic parametrization of finite loop spaces and the
solution of the maximal torus conjecture.
Section~\ref{section:steenrod} presents the solution of the Steenrod problem
of realizing polynomial cohomology rings, and
finally
Section~\ref{section:vistas} provides brief samples of other topics in
homotopical group theory.

\smallskip\noindent
{\em Notation:}
Throughout this paper, the word ``space'' will mean ``topological space of the
homotopy type of a CW--complex''. 

\smallskip\noindent
{\em Acknowledgments:} I would like to thank Kasper Andersen, Bill Dwyer, Haynes
Miller, and Bob Oliver 
for providing helpful comments on a preliminary
version of this paper. I take the opportunity to thank
my coauthors on the various work reported on here, and in particular
express my gratitude to Kasper Andersen for
our mathematical collaboration and sparring through the years.

\section{Root data over the $p$--adic integers}\label{section:root}
In standard Lie theory, root data classify compact
connected Lie groups as well as reductive algebraic groups over algebraically
closed fields. A root
datum is usually packaged as a quadruple $(M,\Phi,M^{\vee},\Phi^\vee)$ of roots $\Phi$ and coroots $\Phi^\vee$ in a $\Z$--lattice $M$
and its dual $M^\vee$, satisfying some conditions
\cite{demazure62short}. For
$p$--compact groups the lattices that come up are 
lattices over
the $p$--adic integers $\Z_p$, rather than $\Z$, so the concept of a root datum needs to 
be tweaked to make sense also in this setting, and one must carry out a
corresponding classification.  In this section we produce a short
summary of this
theory, based on \cite{notbohm96short,DW05,AG08auto,AG09}. 
In what follows $R$ denotes a principal ideal domain of characteristic
zero. 

\medskip

The starting point is the theory of reflection
groups, surveyed e.g.\ in \cite{GM06}.
A finite $R$--reflection group
is a pair $(W,L)$ such that $L$ is a finitely generated free
$R$--module and $W \subseteq \Aut_R(L)$ is a finite
subgroup generated by reflections, i.e., non-trivial elements $\sigma$ that fix an
$R$--submodule of corank one.

Reflection groups have been classified for several choices of $R$, the most well-known
cases being the classification of finite real and rational reflection
groups in terms of certain Coxeter diagrams \cite{humphreys90}. Finite
complex reflection
groups were classified by Shephard--Todd \cite{ST54} in 1954. 
The main (irreducible) examples in the complex case are the groups 
$G(m,s,n)$ of $n\times n$ monomial matrices with non-zero entries
being $m$th roots of unity and determinant an $(m/s)$th root of unity, where
$s | m$; in addition to this
there are $34$ exceptional cases usually named $G_4$ to $G_{37}$.
From the
classification over $\mathbb C$ one
can obtain a classification over $\Q_p$ as the sublist
whose character field $\Q(\chi)$ is embeddable in $\Q_p$. This was
examined by Clark--Ewing \cite{CE74}, and we list their result in
Table~1, using the original notation.

\begin{table}[ht]
{
\tiny
$$\begin{array}{@{}l|l|l|l|l@{}@{}l@{}}
\mbox{W} & \mbox{Order} & \mbox{Degrees} & {\mathbb Q}(\chi) & \mbox{Primes} \\
\hline
\Sigma_{n+1} \mbox{\hspace{1.41cm} (family 1)}& (n+1)!&2,3,\ldots,n+1 & \Q& \mbox{ all } p\\
\hline
G(m,s,n) \mbox{\hspace{0.93cm} (family 2a)} & \multirow{2}{*}{$n!m^{n-1}\frac{m}{s}$}&\multirow{2}{*}{$m,2m,\ldots,(n-1)m, n\frac{m}{s}$}& \multirow{2}{*}{$\Q(\zeta_{m})$}& p \equiv 1  &\pod{m};\\
m \geq 2, n\geq 2, m \neq s \mbox{ if }n=2&&&&
 \multicolumn{2}{|l}{{\mbox{all } p \mbox{ for } m=2}}
\\
\hline
D_{2m} = G(m,m,2) \mbox{\,\,(family 2b)}& \multirow{2}{*}{$2m$}&\multirow{2}{*}{$2,m$}& \multirow{2}{*}{$\Q(\zeta_m +\zeta_m^{-1})$}& p \equiv \pm 1 &\pod{m};\\
 m\geq3  &&&&  \multicolumn{2}{|l}{{\mbox{all } p \mbox{ for } m=3,4,6}}\\
\hline
C_m = G(m,1,1)  \mbox{\,\,\,\,\,\, (family 3)}& \multirow{2}{*}{$m$}& \multirow{2}{*}{$m$}& \multirow{2}{*}{$\Q(\zeta_m)$}&p \equiv 1 &\pod{m};\\
 m\geq2&&&&
 \multicolumn{2}{|l}{{\mbox{all } p \mbox{ for } m=2}}
\\
\hline
  G_4 & 24 & 4,6 & {\mathbb Q}(\zeta_3) & p \equiv 1 &\pod{3}  \\
  G_5 & 72 & 6,12 & {\mathbb Q}(\zeta_3) & p \equiv 1 &\pod{3}  \\
  G_6 & 48& 4,12 & {\mathbb Q}(\zeta_{12}) & p \equiv 1 &\pod{12}  \\
  G_7 & 144& 12,12 & {\mathbb Q}(\zeta_{12}) & p \equiv 1 &\pod{12} \\
  G_8 & 96& 8,12 & {\mathbb Q}(\zeta_{4}) & p \equiv 1 &\pod{4} \\
  G_9 & 192 & 8,24 & {\mathbb Q}(\zeta_8) & p \equiv 1 &\pod{8} \\
  G_{10} & 288& 12,24 & {\mathbb Q}(\zeta_{12}) & p \equiv 1 &\pod{12} \\
  G_{11} & 576 & 24,24 & {\mathbb Q}(\zeta_{24}) & p \equiv 1 &\pod{24} \\
  G_{12} &  48& 6,8 & {\mathbb Q}(\sqrt{-2}) & p \equiv 1,3 &\pod{8} \\
  G_{13} & 96 & 8,12 & {\mathbb Q}(\zeta_8) & p \equiv 1 &\pod{8} \\
  G_{14} & 144 & 6,24 & {\mathbb Q}(\zeta_3,\sqrt{-2}) & p \equiv 1,19 &\pod{24} \\
  G_{15} & 288 & 12,24 & {\mathbb Q}(\zeta_{24}) & p \equiv 1 &\pod{24} \\
  G_{16} & 600 & 20,30 & {\mathbb Q}(\zeta_5) & p \equiv 1 &\pod{5} \\
  G_{17} & 1200& 20,60 & {\mathbb Q}(\zeta_{20}) & p \equiv 1 &\pod{20} \\
  G_{18} & 1800 & 30,60 & {\mathbb Q}(\zeta_{15}) & p \equiv 1 &\pod{15} \\
  G_{19} & 3600& 60,60 & {\mathbb Q}(\zeta_{60}) & p \equiv 1 &\pod{60}  \\
  G_{20} & 360& 12,30 &{\mathbb Q}(\zeta_3,\sqrt{5}) & p \equiv 1,4 &\pod{15}  \\
  G_{21} & 720 & 12,60 & {\mathbb Q}(\zeta_{12},\sqrt{5}) & p \equiv 1,49 &\pod{60} \\
  G_{22} & 240 & 12,20 & {\mathbb Q}(\zeta_4,\sqrt{5}) & p \equiv 1,9 &\pod{20} \\
  G_{23} & 120 & 2,6,10 & {\mathbb Q}(\sqrt{5}) & p \equiv 1,4 &\pod{5}  \\
  G_{24} & 336& 4,6,14 & {\mathbb Q}(\sqrt{-7}) & p \equiv 1,2,4 &\pod{7} \\
  G_{25} & 648 & 6,9,12 & {\mathbb Q}(\zeta_3) & p \equiv 1 &\pod{3} \\
  G_{26} & 1296& 6,12,18 & {\mathbb Q}(\zeta_3) & p \equiv 1 &\pod{3} \\
  G_{27} & 2160 & 6,12,30 & {\mathbb Q}(\zeta_3,\sqrt{5}) & p \equiv 1,4 &\pod{15}\\
  G_{28} & 1152 & 2,6,8,12 & {\mathbb Q} & \mbox{  all } p \\
  G_{29} & 7680& 4,8,12,20 & {\mathbb Q}(\zeta_4) & p\equiv 1 &\pod{4} \\
  G_{30} & 14400& 2,12,20,30 & {\mathbb Q}(\sqrt{5}) & p \equiv 1,4 &\pod{5} \\
  G_{31} & 64\cdot 6!& 8,12,20,24 & {\mathbb Q}(\zeta_4) & p \equiv 1 &\pod{4} \\
  G_{32} & 216\cdot 6!& 12,18,24,30 & {\mathbb Q}(\zeta_3) & p \equiv 1 &\pod{3}  \\
  G_{33} & 72\cdot 6!& 4,6,10,12,18 & {\mathbb Q}(\zeta_3) & p \equiv 1 &\pod{3}  \\
  G_{34} & 108\cdot 9!& 6,12,18,24,30,42 & {\mathbb Q}(\zeta_3) & p \equiv 1 &\pod{3} \\
  G_{35} & 72\cdot 6!& 2,5,6,8,9,12 & {\mathbb Q} & \mbox{  all } p \\
  G_{36} & 8\cdot 9! & 2,6,8,10,12,14,18 & {\mathbb Q} & \mbox{  all } p \\
  G_{37} & 192\cdot 10! & 2,8,12,14,18,20,24,30 & {\mathbb Q} & \mbox{  all } p \\
\end{array}$$
\vspace{-0.5cm}
\caption{The irreducible $\Q_p$-reflection groups}
}\end{table}

The ring $\Q_p[L]^W$ of $W$--invariant polynomial functions on $L$ is
polynomial if and only if $W$ is a reflection group, by the
Shephard--Todd--Chevalley theorem  \cite[Thm.~7.2.1]{benson93}; the
column {\em
  degrees} in Table~1 lists the degrees of the generators, and the
number of degrees equals the rank of
$(W,L)$. 
For many $W$, none of the primes listed in the last column divide $\lvert W\rvert$; in
fact this can only happen in the infinite families, and in the
sporadic examples $12$, $24$, $28$, $29$, $31$, and $34$--$37$.
It is a good exercise to look for the Weyl groups of the various
simple compact Lie groups in the table, where they have character field
$\Q$. One may observe that for $p=2$ and $3$ there is only
one {\em exotic} reflection group (i.e., irreducible
with $\Q(\chi) \neq \Q$), namely $G_{24}$ and $G_{12}$ respectively, whereas for
$p\geq5$ there are always infinitely many.

The classification over $\Q_p$ can be lifted to a classification over
$\Z_p$, but instead of stating this now, we proceed directly to root data.

\begin{defn}
An {\em $R$--root datum} $\D$ is a triple $(W,L,\{Rb_\sigma\})$, where $(W,L)$
is a finite $R$--reflection group, and $\{Rb_\sigma\}$
is a collection of rank one submodules of $L$, indexed by the set of
reflections $\sigma$ in $W$, and satisfying that $\im(1-\sigma) \subseteq
Rb_\sigma$ (coroot condition) and  $w (Rb_\sigma) =
Rb_{w\sigma w^{-1}} \text{ for all } w \in W$ (conjugation invariance).
\end{defn}

An {\em isomorphism} of $R$--root data $\phi\co \D
\to \D'$ is defined to be an isomorphism $\varphi\co L \to L'$  such that $\varphi
W \varphi^{-1} = W'$ as subgroups of $\Aut(L')$ and 
$\varphi(Rb_\sigma) = R b'_{\varphi\sigma\varphi^{-1}}$  for every reflection
$\sigma\in W$.
The element $b_\sigma \in L$, determined up to a unit in $R$, is
called the {\em coroot} corresponding to $\sigma$. The coroot condition
ensures that given $(\sigma,b_\sigma)$  we can define a {\em root} $\beta_\sigma\co L \to R$ via the
formula 
\begin{equation}\label{reflectioneqn}
\sigma(x) = x - \beta_\sigma(x)b_\sigma
\end{equation}

The classification of $R$--root data of course depends heavily on $R$. 
For $R=\Z$ root data correspond bijectively to classically defined
 root data $(M,\Phi,M^\vee,\Phi^\vee)$ via the association $
 (W,L,\{\Z b_\sigma\}) \rightsquigarrow (L^*,\{\pm \beta_\sigma \},L,
 \{\pm  b_\sigma \})$. 
One easily checks that $R b_\sigma \subseteq \ker(N)$,
where $N = 1 + \sigma + \ldots + \sigma^{\lvert\sigma\rvert -1}$ is the norm
element, so giving an $R$--root datum with underlying reflection group
$(W,L)$ corresponds to choosing a cyclic $R$--submodule of
$H^1(\langle\sigma\rangle;L)$ for each conjugacy class of reflections
$\sigma$.  It is hence in practice not hard to parametrize all
possible $R$--root data supported by a given finite $R$--reflection
group. For $R=\Z_p$, $p$ odd, reflections have order
dividing $p-1$, hence prime to $p$, so here $\Z_p$--root data coincides with finite
$\Z_p$--reflection groups. For $R = \Z$ or $\Z_2$  the difference between the two notions
only occur for the root data of $\Sp(n)$ and $\SO(2n+1)$, but
due to the ubiquity of $\SU(2)$ and $\SO(3)$ this distinction turns
out to be an important one.
Note that since a root and a coroot $(\beta_\sigma, b_\sigma)$ determine the reflection
$\sigma$ by \eqref{reflectioneqn}, one could indeed have defined a root datum as a set of pairs
$(\beta_\sigma,b_\sigma)$, each determined up to a unit and
subject to certain conditions;
 see also \cite{nebe99}.

The relationship between $\Z_p$--root data and
$\Z$--root data is given as follows.

\begin{thm}[The classification of $\Z_p$--root data, splitting
  version] \label{rootdata-classification}
\InsertTheoremBreak
\begin{enumerate}
\item \label{rdc-part1}
Any $\Z_p$--root datum $\D$ can be written 
as a product $\D \cong (\D_1
\otimes_{\Z} \Z_p) \times \D_2$, where $\D_1$ is a $\Z$--root datum and $\D_2$
is a product of exotic $\Z_p$--root data.
\item \label{rdc-part2}
Exotic $\Z_p$--root data are in $1$-$1$ correspondence with exotic
$\Q_p$--reflection groups via $\D = (W, L, \{\Z_p
b_\sigma\}) \rightsquigarrow (W, L\otimes_{\Z_p} \Q_p)$.
\end{enumerate}
\end{thm}

Define the
{\em fundamental group} as $\pi_1(\D) = L/L_0$, where $L_0 = \sum_\sigma \Z_pb_\sigma$ is
the coroot lattice, and likewise, with
{\em $p$--discrete torus} $\dT = L \otimes \Z/p^\infty$, we
define the
{\em $p$--discrete center} as  $\dZ(\D) = \bigcap_{\sigma}
\ker(\breve\beta_\sigma\co \dT \to \Z/p^\infty)$; compare e.g.\ \cite{bo9short}. It turns out
that $\pi_1(\D) = \dZ(\D) = 0 $ for all exotic root data, and this plays a
role in the proof of the above statement.
If $A$ is a finite subgroup of $\dZ(\D)$, we can define a quotient
root datum $\D/A$ by taking $\dT_{\D/A} = \dT/A$, and hence $L_{\D/A} = \Hom(\Z/p^\infty,\dT/A)$, and defining the
roots and coroots of $\D/A$ via the induced maps.

\begin{thm}[The classification of $\Z_p$--root data, structure
  version]
\label{data-quotient-structure}
\InsertTheoremBreak
\begin{enumerate}
\item \label{dqs-part1} Any $\Z_p$--root datum $\D = (W, L, \{\Z_p b_\sigma\})$
  can be written as a quotient $$\D =  (\D_1 \times \cdots \times \D_n \times (1, L^W,
  \emptyset))/A$$ 
where $\pi_1(\D_i) = 0$ for all $i$,
for a finite central subgroup $A$.
\item \label{dqs-part2}
Irreducible $\Z_p$--root
data $\D$ with $\pi_1(\D) = 0$ are in $1$-$1$ or $2$-$1$ correspondence with non-trivial irreducible $\Q_p$--reflection
groups via $\D \rightsquigarrow (W,L\otimes_{\Z_p} \Q_p)$, the sole identification
being 
$\D_{\Sp(n)}\otimes_\Z \Z_2$  with $\D_{\Spin(2n+1)}\otimes_\Z \Z_2$,
$n\geq3$.
\end{enumerate}
\end{thm}

A main ingredient used to derive the classification of root data from the 
classification of
$\Q_p$--reflection groups is the case-by-case observation that the mod $p$ reduction
of all the exotic reflection groups remain irreducible,
which ensures that any lift to $\Z_p$ is uniquely determined by the
$\Q_p$--representation.

\begin{rem}
It seems that $\Z_p$--root data ought to parametrize some purely algebraic
objects, just as $\Z$--root data parametrize both compact connected Lie groups and
reductive algebraic groups. Similar
structures come up in Lusztig's approach to the representation theory of finite
groups of Lie type, as examined by Bessis, Brou\'e, Malle, Michel, Rouquier, and
others \cite{broue01}, involving mythical objects from the Greek
island of Spetses \cite{malle98short}.
\end{rem}

\section{$p$--compact groups and their classification}\label{section:pcg}
In this section we give a brief introduction to $p$--compact groups,
followed by the statement of the classification theorem, an
outline of its proof, and some of its consequences. Additional
background information on $p$--compact groups can be found in the
surveys \cite{dwyer98short,lannes95bourbaki,moller95,notbohm95survey}.

The first ingredient we need is the theory of $p$--completions.
The $p$--completion construction of
Sullivan \cite{sullivan05} produces for each space $X$ a map $X
\to X\pcom$, which, when $X$ is simply connected and of finite type,
has the property that $\pi_i(X\pcom) \cong \pi_i(X) \otimes \Z_p$ for
all $i$. A
space is called $p$--complete if this map is a homotopy
equivalence. In fact, when $X$ is simply connected and $H_*(X;\F_p)$
is of finite type, then $X$ is $p$--complete if and only if the homotopy
groups of $X$ are finitely generated $\Z_p$--modules. We remark that Bousfield--Kan \cite{bk}
produced a variant on Sullivan's
$p$--completion functor, and for the spaces that occur in this paper
these two constructions agree up to homotopy, so the words
$p$--complete and 
$p$--completion can be taken in either sense.

A finite loop space is a triple $(X,BX,e)$, where $BX$ is
a pointed connected space, $X$ is a finite CW--complex, and $e\co X \to
\Omega BX$ is a homotopy equivalence, where $\Omega$ denotes based loops.  
We will return to finite loop spaces in Section~\ref{section:loop},
but now move straight to their $p$--complete analogs.

\begin{defn}[$p$--compact group \cite{DW94}]
 A $p$--compact group is a triple $(X,BX,e)$, where 
 $BX$ is a pointed, connected, $p$--complete
space, $H^*(X;\F_p)$ is finite, and $e\co X \xrightarrow{\simeq} \Omega BX$ is a homotopy equivalence.
\end{defn}

The loop multiplication on $\Omega BX$ is here the homotopical analog
of a group structure; while standard loop multiplication does
not define a group, it is equivalent in a strong sense
(as an $A_\infty$--space) to a topological group, whose
classifying space is homotopy equivalent to $BX$. We therefore baptise $BX$
the {\em classifying space}, and note that, since all structure can be
derived from $BX$,
one could equivalently have defined a $p$--compact group to be a space
$BX$, subject to the above conditions. 
 The finiteness of $H^*(X;\F_p)$ is to be thought of as a homotopical
version of compactness, and replaces the condition that the underlying
loop space be homotopy equivalent to a finite complex. We will usually refer to a
$p$--compact group just by $X$ or $BX$ when there is little possibility
for confusion.

Examples of $p$--compact groups include of course the $p$--completed classifying space $BG\pcom$ of a compact
Lie group $G$ with $\pi_0(G)$ a $p$--group. However, non-isomorphic
compact Lie groups may give rise to equivalent $p$--compact groups if they have the
same $p$--local structure, perhaps the most interesting example being
$B\!\SO(2n+1)\pcom \simeq B\!\Sp(n)\pcom$ for $p$ odd
\cite{friedlander75}. Exotic examples 
(i.e.\ examples with exotic root data) 
are discussed in Section~\ref{section:existence}.

A {\em morphism} between $p$--compact groups is a pointed map $BX \to BY$; it
is called a {\em monomorphism} if the homotopy fiber, denoted $Y/X$,
has finite $\F_p$--homology. Two morphisms are called {\em conjugate} if they are freely
homotopic, and two $p$--compact groups are called {\em isomorphic} if their
classifying spaces are homotopy equivalent.
 A $p$--compact group is called {\em connected} if $X$ is connected. 
By a standard argument $H^*(BX;\Z_p)\otimes \Q$ is seen to be
a $\Q_p$--polynomial algebra, and we define the rank $r = \rank(X)$ to be number
of generators. The following is the main structural result of Dwyer--Wilkerson
\cite{DW94}.
\begin{thm}[Maximal tori and Weyl groups of $p$--compact groups \cite{DW94}] \label{pcgprops}
\InsertTheoremBreak
\begin{enumerate}
\item \label{pcgone} Any $p$--compact group $X$ has a {\em  maximal torus}: a
  monomorphism $i\co BT = (BS^1\pcom)^r \to BX$ with $r$ the rank of $X$. Any
  other monomorphism $i'\co BT' = (BS^1\pcom)^s \to BX$ factors as $i' \simeq
  i \circ \phi$ for some $\phi\co BT' \to BT$. In particular $i$ is
  unique up to conjugacy.
\item \label{pcgtwo} The {\em Weyl space} $\W_X(T)$, defined as the
topological monoid of self-equivalences $BT \to BT$ over $i$ (with $i$
made into a fibration), has contractible components.
\item\label{pcgthree}  If $X$ is connected, the natural action of the {\em Weyl group} $W_X(T) =
\pi_0({\W}_X(T))$  on $L_X =
\pi_2(BT)$ gives a faithful
representation of $W_X$ as a finite $\Z_p$--reflection group.
\end{enumerate}
\end{thm}
A short outline of the proof can be found in \cite{lannes95bourbaki}. 
The {\em maximal torus normalizer} is defined as the homotopy orbit space, or Borel
  construction, $B\N_X(T) = BT_{h\W_X(T)}$
and hence sits in a fibration sequence
$$BT \to B\N_X(T) \to B\W_X(T).$$
The normalizer is said to be {\em
  split} if the above fibration has a section.
It is
worth mentioning that one sees that $(W_X,L_X)$ is a $\Z_p$--reflection
group indirectly, by proving that
$$H^*(BX;\Z_p) \otimes \Q \cong (H^*(BT;\Z_p) \otimes \Q)^{W_X}.$$

To define the $\Z_p$--root datum, one therefore needs to proceed differently \cite{DW05,AG08auto,AG09}.
For $p$ odd, the $\Z_p$--root datum $\D_X$ can be defined from the
$\Z_p$--reflection group $(W_X,L_X)$, by setting $\Z_pb_\sigma = \im(L_X \xrightarrow{1 - \sigma} L_X)$.
The definition for $p=2$ is slightly more
complicated, and in order to give meaning to the words we first need a few extra definitions
for $p$--compact groups. 
The centralizer of a morphism $\nu: BA \to BX$
is defined as $B\cC_X(\nu) = \map(BA,BX)_\nu$, where the subscript means the
component corresponding to $\nu$. While this may look odd
at first sight, it does in fact generalize the Lie group notion \cite{DZ87}.
 For a connected $p$--compact group $X$ define the derived $p$--compact group
$\calD X$ to be the covering space of $X$ corresponding to the torsion
subgroup of $\pi_1(X)$.
Consider the {\em $p$--discrete singular torus ${\dT^{\langle \sigma \rangle}}_0$ for
$\sigma$}, i.e., the largest divisible subgroup of the fixed-points
$\dT^{\langle \sigma \rangle}$, with $\dT = L_X \otimes \Z/p^\infty$, and set
$X_\sigma = \calD (\cC_X(\dT_0^{\langle \sigma\rangle}))$. Then $X_\sigma$ is
a connected $p$--compact group of rank
one with $p$--discrete maximal torus $(1-\sigma)\dT$; denote the corresponding maximal torus
normalizer by $\N_\sigma$, called the {\em root subgroup} of $\sigma$. Define the coroots via the formula
$$
\Z_pb_\sigma = \tuborg \im(L_X \xrightarrow{1-\sigma} L_X) & \mbox{if } \N_\sigma
 \mbox{ is split} \\
\ker(L_X \xrightarrow{1+\sigma} L_X) & \mbox{if } \N_\sigma \mbox{ is not
  split}
\sluttuborg
$$
Only the first case happens for $p$ odd, and for $p=2$ the split case corresponds to $BX_\sigma \simeq B\!\SO(3)\twocom$ and
the non-split corresponds to $BX_\sigma \simeq
B\!\SU(2)\twocom$. For comparison we note that when $X =
G\pcom$, for a reductive complex algebraic group $G$, $BX_\sigma \simeq B\langle
U_\alpha,U_{-\alpha}\rangle \pcom$, where $U_\alpha$ is what is ordinarily
called the root subgroup of the root $\alpha = \beta_\sigma$, and the above
formula can be read off from e.g.\
\cite[Pf.~of~Lem.~7.3.5]{spr:linalggrp}.  

We can now state the classification theorem.

\begin{thm}[Classification of $p$--compact groups {\cite{AGMV08,AG09}}] \label{conn-classification}
The assignment which to a connected $p$--compact group $X$ associates
its $\Z_p$--root datum $\D_X$ gives a one-to-one correspondence
between connected $p$--compact groups, up to isomorphism, and
 $\Z_p$--root data, up to isomorphism.

Furthermore the map $\Phi\co \Out(BX)  \to \Out(\D_X)$, given by lifting
a self-homotopy equivalence of $BX$ to $BT$, is an
isomorphism. 
\end{thm}

Here $\Out(BX)$ denotes the group of free homotopy classes of self-homotopy
equivalences $BX \to BX$, and $\Out(\D_X)= \Aut(\D_X)/W_X$. A stronger space-level statement about self-maps
is in fact true, namely 
\begin{equation}
B\!\Aut(BX) \xrightarrow{\simeq} ((B^2\dZ(\D_X))\pcom)_{h\Out(\D_X)}
\end{equation}
where $\Aut(BX)$ is the space of self-homotopy equivalences, $\dZ(\D_X)$ the
$p$--discrete center of $\D_X$ as introduced in Section~\ref{section:root}, and the action of $\Out(\D_X)$ on
$(B^2\dZ(\D_X))\pcom$ is the canonical one. Having control of the whole space
of self-equivalences turns out to be important in the proof.

Theorem~\ref{conn-classification} implies, by
Theorem~\ref{rootdata-classification}, that any connected $p$--compact
group splits as a product of the $p$--completion of a compact connected Lie group
and a product of known exotic $p$--compact groups. For $p=2$ it
shows that there is only one exotic $2$--compact group, the one corresponding to the
$\Q_2$--reflection group $G_{24}$, and this $2$--compact group was constructed
in \cite{DW93}. We will return to the construction of the exotic $p$--compact groups
in the next subsection.

Since we understand the whole space of self--equivalences, one can derive a
classification also of non-connected $p$--compact groups. The set of
isomorphism classes of non-connected $p$--compact groups with root datum of
the identity component $\D$ and group of components $\pi$, is parametrized by
the components of the
moduli space
\begin{equation}
\map(B\pi,((B^2\dZ(\D))\pcom)_{h\Out(\D)})_{h\Aut(B\pi)}
\end{equation}

As with the classification of compact Lie groups, the classification statement can
naturally be broken up into two parts, existence and uniqueness of
$p$--compact groups. The
uniqueness statement can be formulated as an {\em isomorphism theorem} saying that there is a $1$-$1$--correspondence between
conjugacy classes of isomorphisms of connected $p$--compact groups $BX \to BX'$ and
isomorphisms of root data $\D_X \to \D_{X'}$, up to
$W_{X'}$--conjugation. This last statement can in fact be strengthened to an {\em isogeny
theorem} classifying maps that are rational isomorphisms \cite{AG10isogeny}.

While the existence and uniqueness are
separate statements, they are currently most succinctly proved
simultaneously by an induction on the size of $\D$, since the proof of
existence requires knowledge of certain facts about self-maps, and the proof
of uniqueness at the last step is aided by specific facts about concrete
models. We will discuss the proof of existence in
Section~\ref{section:existence} and of uniqueness in
Section~\ref{section:uniqueness}, along with some information about the
history.

\subsection{Construction of $p$--compact groups}\label{section:existence} 
Compact connected Lie groups can be constructed in different
ways. They can be exhibited as symmetries of geometric
objects, or can be systematically constructed via
generators-and-relations type constructions that involve first constructing a finite
dimensional Lie algebra from the root system, and then passing to the group \cite{kac90,spr:linalggrp}.

An adaptation of the above tools to $p$--compact groups is still
largely missing, so one currently has to proceed by more ad hoc
means, with the limited aim of constructing only the exotic $p$--compact
groups. These were in fact
already constructed some years ago, but we take the opportunity here
to retell the tale, and outline the closest we currently get to a streamlined construction.

The first exotic $p$--compact groups were constructed by Sullivan
\cite{sullivan05} as the homotopy
orbit space of the action of the would-be Weyl group on the
would-be torus. The most basic case he observed is
the following: If $C_m$ is a cyclic group of order $m$, and $p$
an odd prime such that $m | p-1$, then $C_m \leq \Z_p^\times$, and hence $C_m$ acts on the
Eilenberg--MacLane space $K(\Z_p,2)$. The Serre spectral sequence for the fibration 
$$K(\Z_p,2) \to
K(\Z_p,2)_{hC_m} \to BC_m $$
reveals that the $\F_p$--cohomology of $K(\Z_p,2)_{hC_m}$ is a
polynomial algebra on a class in degree $2m$, using that $m$ is prime to $p$. Therefore the cohomology of its loop space is an exterior algebra
in degree $2m-1$ and
$BX = (K(\Z_p,2)_{hC_m})\pcom$ is a $p$--compact group, with $\Omega BX \simeq
(S^{2m-1})\pcom$. We have just realized all exotic groups in family 3 of
Table~1!

Exactly the same argument carries over to the general case of a root
datum $\D$ where $p \nmid \lvert W\rvert$, just replacing $C_n$ by $W$ and $\Z_p$
by $L$, since
 $\F_p[L\otimes \F_p]^W$ is a polynomial
algebra exactly when $W$ is a reflection group, when $p \nmid \lvert W\rvert $, by the
Shephard--Todd--Chevalley theorem used earlier.
 This observation was made by
Clark--Ewing \cite{CE74}, and realizes a large number of
groups in Table~1. However, the method as it
stands cannot be pushed further, since the assumption that
 $p \nmid \lvert W\rvert $ is crucial for the collapse of the Serre spectral sequence.

Additional exotic $p$--compact groups were constructed in the 1970s by other methods.
Quillen realized $G(m,1,n)$ at all possible primes by constructing an
approximation via classifying spaces of discrete groups
\cite[\S10]{quillen72}, 
and Zabrodsky \cite[4.3]{zabrodsky84} realized $G_{12}$
and $G_{31}$ at $p=3$ and $5$ respectively, by taking
homotopy fixed-points of a $p'$--group acting on the classifying space of a
compact Lie group.

To build the remaining exotic $p$--compact
groups one needs a far-reaching generalization of Sullivan's technique, obtained by
replacing the homotopy orbit space with a more sophisticated homotopy
colimit, that ensures that we still get a collapsing spectral sequence even when $p$ divides the order of
$W$. The technique was introduced by Jackowski--McClure \cite{JM92}, as
a decomposition technique in terms of centralizers of elementary abelian
subgroups, and was subsequently used by Aguad\'e
\cite{aguade89} ($G_{12}, G_{29}, G_{31}, G_{34}$), 
Dwyer--Wilkerson \cite{DW93} $(G_{24})$, and
Notbohm--Oliver \cite{notbohm98} ($G(m,s,n)$) to finish the construction of
the exotic $p$--compact groups.

The following is an extension of Aguad\'e's argument, and can be used
inductively to realize all exotic $p$--compact groups for $p$
odd---that this works in all cases relies on the stroke of luck, checked case-by-case, that
all exotic finite
$\Z_p$--reflection groups for $p$ odd have $\Z_p[L]^W$ a
polynomial algebra.

\begin{thm}[Inductive construction of $p$--compact groups, $p$ odd
  {\cite{AGMV08}}] \label{realization}
Con\-sider a finite
  $\Z_p$--reflection group $(W,L)$, $p$ odd,
  with $\Z_p[L]^W$ a polynomial algebra.

Then $(W_V,L)$ is a again a $\Z_p$--reflection group and
  $\Z_p[L]^{W_V}$ a polynomial algebra, for $W_V$
  the pointwise stabilizer in $W$ of $V \leq L \otimes \F_p$. 

Assume that, for all non-trivial $V$, $(W_V,L)$ is realized by a connected
$p$--compact group $Y_V$ satisfying
the isomorphism part of Theorem~\ref{conn-classification} and
 $H^*(Y_V;\Z_p) \cong \Z_p[L]^{W_V}$ (with $L$ in degree $2$).
Then $V \mapsto Y_V$ extends
to a functor 
$Y\co \A^{op} \to \Spaces$,  where $\A$ has objects
 non-trivial $V \leq L \otimes \F_p$ and morphisms given
by conjugation in $W$, and $$BX =  (\hocolim_{\A^{op}} Y)\pcom$$ is a $p$--compact group
with Weyl group $(W,L)$ and $H^*(BX;\Z_p) \cong \Z_p[L]^W$.
\end{thm}

\begin{proof}[Idea of proof] The statement that $\Z_p[L]^{W_V}$ is a polynomial
  algebra is an extension of Steinberg's fixed-point theorem in the version of Nakajima
  \cite[Lem.~1.4]{nakajima79}. The proof uses Lannes'
  $T$--functor, together with case-by-case considerations.

The inductive construction is straightforward, given
current technology, and uses only general arguments: Since we assume
we know $Y_V$ and its automorphisms for all $V \neq 1$, one easily sets up
a functor $\A^{op} \to \ho{\Spaces}$, the homotopy category of spaces, and the task is to rigidify this to a functor in
  the category of spaces. The  diagram can be show to be
  ``centric'', so one can use the obstruction theory developed by
  Dwyer--Kan in \cite{DK92}. The relevant obstruction groups identify
  with the higher limits of a functor obtained by taking fixed-points,
  and in particular form a Mackey functor whose higher
  limits vanish by a theorem of Jackowski--McClure \cite{JM92}. 
We can therefore rigidify the diagram to a
diagram in spaces, and the resulting homotopy colimit is easily shown to have
the desired cohomology.
\end{proof}

We now turn to the prime
$2$. Here the sole exotic $\Z_2$--reflection group is $G_{24}$, 
and the corresponding $2$--compact group was realized by Dwyer--Wilkerson \cite{DW93}
and dubbed $\DI(4)$, due to the fact that, for $E = (\Z/2)^4$,
$$H^*(B\!\DI(4);\F_2) \cong \F_2[E]^{\GL(E)}$$
the rank four Dickson invariants. At first glance this might look like
the setup of Theorem~\ref{realization}, but note that $G_{24}$ is a rank
three $\Z_2$--reflection group, not four, so $E$ is not just the elements
of order $2$ in the maximal torus. However by taking $\A$ to be the
category with objects the non-trivial subgroups of $E$, and morphisms induced by conjugation in $\GL(E)$, and correctly guessing the centralizers of elementary abelian subgroups, the
argument can still be pushed through; the starting point is declaring the
centralizer of any element of order two to be $\Spin(7)\twocom$.

We again stress the apparent luck in being able to guess the rather
uncomplicated structure of $\A$ and the centralizers.
If one hypothetically had to
construct an exotic $p$--compact group with a seriously complicated
cohomology ring, say one would try to construct $E_8$ at the prime $2$ by
these methods, it would not be clear how to start. As
a first step one would need a way to describe the $p$--fusion in the group, just from
the root datum $\D$. This relates to old questions in Lie theory, which
have occupied Borel, Serre, and many others \cite{serre00bourbaki}\ldots

\subsection{Uniqueness of $p$--compact groups}\label{section:uniqueness}
In this subsection, we outline the proof of the uniqueness part of
the classification theorem for $p$--compact groups,
Theorem~\ref{conn-classification}, following \cite{AG09} by Andersen
and the author; it extends \cite{AGMV08} also with M{\o}ller and Viruel. 
We mention that the quest for uniqueness was initiated by Dwyer--Miller--Wilkerson \cite{DMW86}
in the 80s and in particular Notbohm \cite{notbohm94} obtained strong partial
results; a different approach for $p=2$ using
computer algebra was independently given by M{\o}ller
\cite{moller07del1,moller07del2}. See \cite{AGMV08,AG09} for more details on
the history of the proof.

From now on we consider two connected $p$--compact
groups $X$ and $X'$ with the same root datum $\D$, and want to build
a homotopy equivalence $BX \to BX'$.
The proof goes by an induction on the size of $(W,L)$.

\smallskip\noindent
{\em Step $1$: (The maximal torus normalizer and its automorphisms, \cite{DW05,AG08auto}).}
A first step is to show that $X$ and $X'$ have isomorphic maximal torus
normalizers. Working with the maximal torus normalizer has a number of technical
advantages over the maximal torus, related to the fact that the fiber of the map
$B\N \to BX$ has Euler characteristic prime to $p$ (one, actually).

One shows that the maximal torus normalizers are isomorphic, by giving a {\em
  construction} from the root datum. For $p$ odd
the construction is simple, since the maximal torus normalizer turns out
always to be split, and hence isomorphic to $(B^2L)_{hW}$ with the canonical
action.
This was established in \cite{kksa:thesisshort}, by showing that the relevant
extension group is zero except in one case, which can be handled by
other means; cf.\ also \cite[Rem.~2.5]{AGMV08}. For $p=2$, the problem is more
difficult. The corresponding problem for compact Lie groups, or reductive
algebraic groups, was solved by Tits \cite{tits66} many years ago. A
thorough reading of Tits' paper, with a cohomological rephrasing of
some of his key constructions, allows his construction to
be pushed through also for $p$--compact groups
\cite{DW05}. One thus algebraically constructs a maximal torus normalizer
$\N_\D$ and show it to be isomorphic to the topologically defined one. A
problem is however that $\N$ in general has too many automorphisms. To
correct this, it was shown in \cite{AG08auto} that the root subgroups
$\N_\sigma$,
introduced before Theorem~\ref{conn-classification}, can also be built algebraically, and
adding this extra data give the correct automorphisms. Concretely, one has a
canonical factorization 
$$\Phi\co \Out(BX) \to \Out(B\N,\{B\N_\sigma\}) \xrightarrow{\cong} \Out(\D_X)$$
and one can furthermore build a candidate model
for the whole space $B\!\Aut(BX)$,  by a
slight modification of $B\!\Aut(B\N,\{B\N_\sigma\})$, the space of
self-homotopy equivalences of $B\N$ preserving the root subgroups.

\smallskip\noindent
{\em Step $2$: (Reduction to simple, center-free groups, \cite[\S2]{AG09}).} This
next step involves relating the $p$--compact group and its summands and center-free
quotient via certain fibration sequences, and studying
automorphisms via these fibrations.
 Several of
the necessary tools, such as the understanding of the center of a
$p$--compact group \cite{dw:center}, the product splitting theorem \cite{dw:split}, etc., were
already available in the 90s. But, in particular for $p=2$, one needs to
incorporate the machinery of root data and root
subgroups; we refer the reader to \cite[\S2]{AG09} for the details.

\smallskip\noindent
{\em Step $3$: (Defining a map on centralizers of elements of order $p$, {\cite[\S4]{AG09}}).}
We now assume that $X$ and $X'$ are simple, center-free $p$--compact
groups. The next tool needed is a homology decomposition theorem, more
precisely the
centralizer decomposition, of Jackowski--McClure \cite{JM92} and Dwyer--Wilkerson
\cite{DW92}, already mentioned in the previous subsection. Let $\A(X)$ be the Quillen category of $X$ with objects monomorphisms $\nu\co BE
\to BX$, where $E = (\Z/p)^s$ is a non-trivial elementary abelian $p$--group,
and the morphisms $(\nu\co BE
\to BX) \to
(\nu'\co BE' \to BX)$ are the group monomorphisms $\phi\co E \to E'$ such that $\nu' \circ B\phi$ is
conjugate to $\nu$. The centralizer decomposition
theorem now says that for any $p$--compact group $X$, the evaluation map
$$\hocolim_{\nu \in \A(X)^{\op}} B\cC_X(\nu) \to BX$$
is an isomorphism on $\F_p$--cohomology.

This opens the possibility for a proof by induction, since the centralizers
will be smaller $p$--compact groups if $X$ is
center-free. 
As explained above we can assume that $X$ and $X'$ have common
maximal torus normalizer and root subgroups
$(\N,\{\N_\sigma\})$, so that we are in the situation of
following diagram
$$\xymatrix{
 &(B\N,\{B\N_\sigma\}) \ar[dl]_j \ar[dr]^{j'}\\
BX \ar@.[rr] & & BX'}$$
where the dotted arrow is the one we want to construct.

If $\nu\co B\Z/p \to BX$ is a monomorphism, then it can
be conjugated into $T$, uniquely up to conjugation in $\N$. This
gives a well defined way of viewing $\nu$ as a map $\nu\co B\Z/p \to
BT \to B\N$. 
Taking centralizers of this map produces a new diagram
$$\xymatrix{
& (B\cC_\N(\nu),\{B\cC_\N(\nu)_\sigma\}) \ar[dr] \ar[dl]\\
B\cC_X(\nu) \ar@.[rr]&&B\cC_{X'}(\nu)}$$

One now argues that the induction hypothesis guarantees
that we can construct the dotted arrow. There is the slight twist that the
centralizer will be disconnected in general, so we have to use that we inductively know the whole space of
self-equivalences of the identity component.

\smallskip\noindent
{\em Step $4$: (Compatibility of maps on all centralizers, \cite[\S5]{AG09}).}
The next step is to define the map on centralizers of arbitrary elementary
abelian $p$--subgroups $\nu\co BE \to BX$. This is done by restricting to a
rank one subgroup $E' \leq E$ and considering the composition
$$B\cC_X(\nu) \to B\cC_X(\nu|_{E'}) \to B\cC_{X'}(\nu|_{E'}) \to BX'.$$
One now has to show that these maps do not depend on the choice of $E'$, and that they fit together to define an element in
$$\lim_{\nu \in \A(X)}\!\!\!\!{}^0[B\cC_X(\nu),BX']$$
By using induction it turns out that one can reduce to the case where $E$ has rank two and
$\cC_X(\nu)$ is discrete. An inspection of the classification of
$\Z_p$--root data shows that this case only occurs for $\D \cong
\D_{\PU(p)\pcom}$, which can then be handled by direct arguments, producing the
element in $\lim^0$.

In fact one can prove something slightly stronger, which will be needed in the
next step: A close inspection of the
whole preceeding argument reveals that all maps can be constructed over $B^2\pi_1(\D)$, which allows one to
produce an element in
$$\lim_{\nu \in \A(X)}\!\!\!\!{}^0[B\widetilde{\cC_X(\nu)},B\widetilde{X'}]$$
where the tilde denotes covers with respect to the kernel of
the map to $\pi_1(\D)$.

With this step complete one can see that $BX$ and $BX'$ have the same $p$--fusion,
i.e., that $p$--subgroups are conjugate in the same way, but we are
left with a rigidification issue.

\smallskip\noindent
{\em Step $5$: (Rigidifying the map, \cite[\S6]{AG09}).}
One now wants to define a map on the whole homotopy colimit, which can then
easily be checked to have the correct properties, finishing the proof of the classification.
Constructing such a map directly from an element in $\lim^0$ requires
knowing that the higher limits of the functors  $F_i\co\A(X) \to
\Z_p\mbox{-}\operatorname{mod}$ given by $E \mapsto \pi_i(\cZ\cC_X(E))$,
vanish, where $\cZ$ denotes the center.
In turn, this calculation requires knowing the structure of $\A(X)$, and for this
we use that $X$ is a known $p$--compact group, where we can examine the
structure. For the part of the functor corresponding to elementary abelian
subgroups that can be conjugated into $T$, the higher limits can be show to vanish via a Mackey
functor argument, going back to \cite{JM92} and \cite{DW92}. This in fact equals the whole functor for all exotic groups for $p$ odd, and
$\DI(4)$ also works via a variant on this argument, which finish off
those cases.

We can hence assume that $X$ is the $p$--completion of a compact connected Lie group $G$.
Here the obstruction groups were computed to identically vanish in \cite{AGMV08}, for $p$
odd, relying on detailed information about
the elementary abelian $p$--subgroups of $G$, partially tabulated by Griess
\cite{griess}. This is easy when there is little torsion in the cohomology, but
harder for the small torsion primes, and the exceptional groups.
In \cite{AG09}, however, we use a different argument to cover all primes.
 Using the above element in $\lim^0$
it turns out that one can produce an element in
$$\lim_{\tilde G/\tilde P \in \bO^r_p(\tilde G)^{\op}}\!\!\!\!\!\!\!\!\!{}^0[B\tilde P,B\widetilde{X'}]$$
where $\bO^r_p(\tilde G)$ is the subcategory of the orbit category of $G$ with
objects the so-called $p$--radical subgroups. Here one again wants to show
vanishing of the higher limits, in order to get a map on the homotopy colimit. Calculating higher limits over this orbit category is in many ways
similar to calculating it over the Quillen category \cite{grodal02}. In this
case, however, the relevant higher limits were in fact shown to identically vanish in earlier work of
Jackowski--McClure--Oliver \cite{JMO92}, also building on substantial
case-by-case calculations.
This again produces a map $B\tilde G \xrightarrow{\simeq} B\widetilde{X'}$, and
passing to a quotient provides the sought homotopy
equivalence  $BG \xrightarrow{\simeq} BX'$. The
statements about self-maps also fall out of this approach.\qed

\subsection{Lie theory for $p$--compact groups}\label{section:lietheory}
We have already seen many Lie-type results for $p$--compact groups.
Quite a few more can be proved by observing
that the classical Lie result only depends on the $p$--completion of
the compact Lie group, and
verifying case-by-case that it holds for the exotic $p$--compact groups.
We collect some theorems of this type in this section, encouraging the
reader to look for more conceptual proofs, and include also a brief discussion of
homotopical representation theory. Throughout this section $X$ is a
connected $p$--compact group with maximal torus $T$.

The first theorem on the list is the analog of theorems of Bott \cite{bott56}
from $1954$.

\begin{thm}\label{homogeneous} 
$H^*(X/T;\Z_p)$ and $H^*(\Omega X;\Z_p)$ are both torsion free and
  concentrated in
even degrees, and  $H^*(X/T;\Z_p)$ has rank $\lvert W_X\rvert $ as a $\Z_p$--module.
\end{thm}

The result about $\Omega X$ was known as the
loop space conjecture, and in fact proved by Lin and Kane in a series of
papers in the
more general setting of finite mod $p$ $H$--spaces, using complicated
calculations with Steenrod operations
\cite{lin82}.

Bott's proof used Morse theory and the result may be viewed in
the context of Schubert cell decompositions \cite{mitchell88}. 
Rationally
$H^*(X/T;\Z_p)\otimes \Q = \Q_p[L] \otimes_{\Q_p[L]^{W_X}} \Q_p$, so calculating
the Betti numbers, given the theorem, is reduced to a question about complex reflection
groups---an interpretation of these numbers in terms of length functions on the root
system has been obtained for certain classes of complex
reflection groups, cf.\ \cite{BM98,totaro03}, but the complete
picture is still not clear.
In general the theory of homogeneous
and symmetric spaces for $p$--compact groups is rather unexplored, and warrants attention.

Theorem~\ref{homogeneous} implies that $\pi_3(X)$ is torsion free, and proving that in a
conceptual way might be a good starting point. For Lie groups, Bott in fact
stated the, now classical, fact that $\pi_3(G) \cong \Z$ for $G$ simple. The analogous statement is {\em not} true
for most of the exotic $p$--compact groups; for instance it obviously fail for
the Sullivan spheres
other than $S^3$. However, it {\em is} true that $\pi_3$ is non-zero for finite loop spaces, as a
consequence of a celebrated theorem of Clark \cite{clark63} from $1963$ giving strong
restrictions on the degrees of finite loop spaces. These results helped fuel the
speculation that finite loop spaces should look a lot like compact Lie groups,
a point we will return to in the next section.

\smallskip

Most of the general results about torsion in the cohomology of $BX$ and $X$
due to Borel, Steinberg, and others, also carry through to $p$--compact groups,
but here again with many results relying on the classification. This fault is partly
inherited from Lie groups; see Borel \cite[p.~775]{boreloeuvresII} for a
summary of the status there. In particular
we mention that $X$ has torsion free $\Z_p$--cohomology if and only if $BX$ has torsion
free $\Z_p$--cohomology if and only if every elementary abelian $p$--subgroup factors
through a maximal torus. Likewise $\pi_1(X)$ is torsion free if and only if
every elementary abelian group of rank two factors through a maximal torus; see
\cite{AGMV08,AG09}.

\smallskip

The (complex linear) homotopy representation theory of $X$ is encoded
in the semi-ring
$$\Rep^\bbC(BX) = [BX,\coprod_n B\U(n)\pcom]$$
It is non-trivial since for any connected $p$--compact group $X$
there exists a mono\-morphism
  $BX \to B\U(n)\pcom$, for some $n$; the exotic groups were checked in 
\cite{castellanaexotic-thesis,castellana06,ziemianski09}---indeed, as already alluded to,
several exotic $p$--compact groups can conveniently be {\em constructed} as
homotopy fixed-points inside a $p$--completed compact Lie group.
The general structure of the semi-ring is however still far from
understood. The classification allows one to focus on $p$--completed
classifying spaces of compact Lie groups, but even in this case the semi-ring appears very complicated
\cite{JMO95revisited}; there are higher limits
obstructions, related to interesting problems in
group theory \cite{grodal02}.

Weights can be constructed as usual:
By the existence of a maximal torus, we can lift a homotopy representation
to a map $BT_X \to BT_{\U(n)\pcom}$, well defined up to an action of $\Sigma_n$, and produce a collection of $n$ weights
in $L_X^* = \Hom_{\Z_p}(L_X,\Z_p)$, invariant under the action of the Weyl
group $W_X$.
When $p \nmid \lvert W_X\rvert $, homotopy representations just
correspond to finite $W_X$--invariant
subsets of $L_X^*$, and any homotopy representation decomposes up to
conjugation uniquely into indecomposable representations given by transitive
$W_X$--sets.
When $p \mid \lvert W_X\rvert $ the situation is much more
complicated.

Let us describe what happens in the basic case of $X = \SU(2)\twocom$.
Denote by $\rho_i$ the irreducible complex representation of $\SU(2)$ with highest
weight $i$, and use the same letter for the induced map $B\!\SU(2)\twocom \to
B\U(i+1)\twocom$. Precomposing with the self-homotopy
equivalence  $\psi^k$ of $B\!\SU(2)\twocom$, $k \in \Z_2^\times$, corresponding
to multiplication by $k$ on the root
datum, gives a new representation $k \star \rho_i$ of the same
dimension, but with weights multiplied by $k$.

\begin{thm}$\Rep^\bbC(B\!\SU(2)\twocom)$ has an additive generating set given by $\rho_0$, $k \star \rho_1$, $k \star
  \rho_2$ and $((k+2k') \star \rho_1) \otimes ((k-2k') \star \rho_1)$,
  $k \in \Z_2^\times$, $0 \neq k' \in \Z_2$. These generators are indecomposable, and two representations agree if they have the same weights.
\end{thm}

The reader may verify that the decomposition into indecomposables is not unique, e.g.\ for $\rho_6$.
It is at present not clear how to use $\SU(2)\twocom$ to describe the
general structure, as one could have hoped---the thing to note is
that homotopy representations are governed by questions of $p$--fusion of
elements, rather than more global structure. Already for
$\SU(2)\twocom \times \SU(2)\twocom$ there is no upper
bound on the dimension of the indecomposables, and in
particular they are not always a tensor product of indecomposable $\SU(2)\twocom$
representations. More severely, representations need not be uniquely
determined by their weights, e.g.\ for $\Sp(2)\twocom \times \Sp(2)\twocom$.

By using case-by-case arguments, there might be hope to establish a version of Weyl's theorem $R(BX) \xrightarrow{\cong}
R(BT)^{W_X}$, where $R(BX) =
\Gr(\Rep^\C(BX))$ is the Grothendieck group. The result is not proved even for
$p$--completions of compact Lie groups, but the integral version is the main
result in \cite{JO96}. The weaker $K$--theoretic result $K^*(BX;\Z_p) \xrightarrow{\cong}
K^*(BT;\Z_p)^W$ was established in \cite{JO97} (using that $H^*(\Omega X;\Z_p)$
is torsion free). The ring structure of $R(BT)^W$ is also not clear,
and in particular it would be interesting to exhibit some fundamental representations.

\section{Finite loop spaces}\label{section:loop}
In the 1960s and early 1970s, finite loop spaces, not $p$--compact groups, were
the primary objects of study, and there were many conjectures about them \cite{stasheff71}.
The theory of $p$--compact groups enables the resolution of most of them, either in the positive or the
negative, and gives what is essentially 
a parametrization of all finite loop spaces.

We already defined finite loop spaces in Section~\ref{section:pcg}; let us now briefly
recall their history in broad strokes.
 Hopf proved in 1941 \cite{hopf41} that the rational cohomology of any connected,
finite loop space is a graded exterior algebra $H^*(X;\Q) \cong
\bigwedge_\Q(x_1,...,x_r)$, where $\lvert x_i\rvert  = 2d_i-1$, and $r$ is called the {\em rank}. Serre, ten years later  \cite{serre53}, 
showed that 
the list of {\em degrees} $d_1,\ldots,d_r$ uniquely determines the
rational homotopy type of $(X,BX,e)$. In those days, there were not
many examples of finite loop spaces. Indeed, in the early 1960s it was speculated that perhaps every
finite loop space was homotopy equivalent to a compact Lie group, a would-be
variant of Hilbert's $5$th problem. 
This was soon shown to be wrong in several different ways:
Hilton--Roitberg, in 1968, exhibited a 'criminal'
\cite{HR69}, a finite loop space $(X,BX,e)$, of the rational
homotopy type of $\Sp(2)$, such that the underlying space $X$ is not homotopy
equivalent to any Lie group; and Rector \cite{rector71loop} in 1971 observed that
there exists {\em uncountable} many finite loop spaces $(X,BX,e)$ such that $X$ is
homotopy equivalent to $\SU(2)$. The first example may superficially look more
benign than the second; indeed in general there are only finitely
many possibilities for the homotopy type of the underlying space $X$, given the
rational homotopy type of $BX$ \cite{CD71}.
But the exact number depends on homotopy
groups of finite complexes, and does not appear closely related to Lie theory,
so shifting focus from loop space structures $(X,BX,e)$ to that of
homotopy types of $X$, does not appear desirable.

An apparently better option is, as the reader has probably sensed, to pass to
$p$--completions, defined in Section~\ref{section:pcg}.
Sullivan made precise how one can recover a (simply connected) space integrally if
one knows the space ``at all primes and rationally, as well as how they
are glued together''. Along with his $p$--completion, he constructed a
rationalization functor $X \to X_\Q$, with analogous properties, and proved
that these functors fit together in the following arithmetic square.
\begin{prop}[Sullivan's arithmetic square \cite{sullivan74,DDK77}]  \label{arithmetic}Let $Y$ be a simply
  connected space of finite type. Then the following diagram, with obvious maps, is
  a homotopy pull-back square.
$$\xymatrix{Y \ar[r] \ar[d] &  {\prod}_p Y\pcom  \ar[d]\\
Y_\Q \ar[r] & (\prod_p Y\pcom)_\Q}$$
\end{prop}

This parallels the usual fact that the integers $\Z$ is a pullback of $\hat \Z
= \prod_p \Z_p$ and $\Q$
over the finite adeles $\bA_f = \hat \Z \otimes \Q$. 
If $BX$ is the classifying
space of a connected finite loop space then, by the classification of
$p$--compact groups, all spaces in the diagram are
now understood:
Each $BX\pcom$  is the classifying space of a $p$--compact
group, and the spaces at the bottom of the diagram are determined by
numerical data, namely the degrees: $BX_\Q \simeq K(\Q,2d_1) \times \cdots \times
K(\Q,2d_r)$ and  $(\prod_p BX\pcom)_\Q \simeq K(\bA_f,2d_1) \times \cdots \times
K(\bA_f,2d_r)$, by the result of Serre quoted earlier.
Hence to classify finite loop spaces with some given degrees,
we first have to enumerate all collections of $p$--compact groups with
those degrees; there are a finite number of
these, and they can be enumerated given the classification
\cite[Prop.~8.18]{AG09}.
The question
of how many finite loop spaces with a given set of $p$--completions
is then a question of {\em genus}, determined by an explicit set of
double cosets.

\begin{thm}[Classification of finite loop spaces]\label{loopclassification}
  The assignment which to a finite loop space $Y$ associates the
  collection of $\Z_p$--root data
  $\{\D_{Y\pcom}\}_p$ is a surjection from connected finite loop spaces to collections of $\Z_p$--root
  data, all $p$, with the same degrees $d_1, \ldots, d_r$.
The pre-image of $\{\D_p\}_p$ is parametrized by the set of double cosets
$$\Out(K_\Q)\backslash \Out^c(K_{\bA_f})/\prod_p \Out(\D_p)$$
where  $K_\F = K(\F,2d_1) \times \cdots \times K(\F,2d_r)$, $\F = \Q$ or $\bA_f$.
\end{thm}

Here $\Out(K_\Q)$ denotes the group of free homotopy classes of
self-homotopy equivalences, and $\Out^c(K_{\bA_f})$ denotes those homotopy equivalences that induce $\bA_f$--linear maps on
homotopy groups. Since $K_\F$ is an Eilenberg--MacLane space, the set of double
cosets can be completely described algebraically;
see~\cite[\S13]{AGMV08} for a calculation of $\Out(\D_p)$.

The set of double cosets will, except for the degenerate case of tori, be
uncountable. Allowing for only a single prime $p$ everywhere above would parametrize the
number of $\Z_{(p)}$--local finite loop spaces corresponding to a given $p$--compact group
$Y_p$, and also this set is usually uncountable, with a few more
exceptions, such as $\SU(2)$. A similar result holds when one inverts some
collection of primes $\mathcal P$; see \cite[Rem.~3.3]{AG08steenrod} for
more information.

\begin{proof}[Sketch of proof of Theorem~\ref{loopclassification}] There is a natural
  inclusion $K_\Q \to K_{\bA_f}$ induced by the unit map $\Q \to \bA_f$, and one easily
  proves that the pull-back provides a space $Y$ such that $H^*(\Omega
  Y;\Z)$ is finite over $\Z$. That $Y$ is actually homotopy equivalent to a
  finite complex follows by the vanishing of the finiteness obstruction, as
  proved by Notbohm \cite{notbohm99} (see \cite[Lemma~1.2]{ABGP04} for
  more details). 
Twisting the pullback by an element in $\Out^c(K_{\bA_f})$ provides a new
  finite loop space, and after passing to sets of double cosets, this assignment is
  easily seen to be surjective and injective on homotopy types (see
   \cite{sullivan74} and \cite[Thm.~3.8]{wilkerson76short}).
\end{proof}

If one assumes that the finite loop space $X$ has a maximal torus, as
defined by Rector \cite{rector71subgroups}, i.e., a map
$(BS^1)^r \to BX$ with homotopy fiber homotopy equivalent to a finite complex,
for $r = \rank(X)$,
the above picture changes completely. The inclusion of an `integral' maximal
torus prohibits the twisting in the earlier theorem, and one obtains a
proof of the
classical maximal torus conjecture stated by Wilkerson \cite{wilkerson74} in 1974, giving a homotopy
theoretical description of compact Lie groups as exactly the finite
loop spaces admitting a maximal torus.

\begin{thm}[Maximal torus conjecture {\cite{AG09}}] \label{maxtorus}
The classifying space functor, which to a compact Lie group
$G$ associates the finite loop space $(G,BG,e\co G \xrightarrow{\simeq}
\Omega BG)$ gives a one-to-one correspondence between isomorphism classes of
compact Lie groups and finite loop spaces with a maximal torus.
Furthermore, for $G$ connected, $\Out(BG) \cong \Out(G) \cong \Out(\D_G)$.
\end{thm}

The statement about automorphisms, which was not part of the original
conjecture, follows from work of Jackowski--McClure--Oliver \cite[Cor.~3.7]{JMO95}.

In light of the above structural statement it is natural to further enquire how exotic finite loop
spaces can be. Whether they are all manifolds was recently settled in the
affirmative by
Bauer--Kitchloo--Notbohm--Pedersen, answering an old question of Browder
\cite{browder61}.

\begin{thm}[\cite{BKNP04}] For any
  finite loop space $(Y,BY,e)$, $Y$ is homotopy equivalent to a
  closed, smooth, parallelizable manifold.
\end{thm}

The result is proved using the theory of $p$--compact groups,
combined with classical surgery techniques, as set up by Pedersen.
It shows the subtle failure of a na\"ive homotopical version of
Hilbert's fifth problem: Every finite loop space is, by classical results, homotopy
equivalent to a topological group, and homotopy equivalent to a
compact smooth manifold by the above. But one cannot always achieve both properties
at once. This would otherwise imply that every finite loop space
was homotopy equivalent to a compact Lie group, by
the solution to Hilbert's fifth problem, contradicting that many
exotic finite loop spaces exist.

One can still ask if every finite loop space is {\em rationally} equivalent to some
compact Lie group? Indeed this was conjectured in the 70s to be the case, and
was verified up to rank
$5$. However, the answer to this question turns
out to be negative as well, although counterexamples only start appearing in
high rank. 

\begin{thm}[A `rational criminal' \cite{ABGP04}] \label{rank66}
There exists a connected finite loop space $X$ of rank $66$, dimension
$1254$, and degrees
$$\{2^8,3^2,4^8,5^2,6^7,7,8^7,9,10^5,11,12^5,13,14^5,16^3,18^2,20^2,22,24^2,26,28,30\}$$
(where $n^k$ means that $n$ is repeated $k$ times)
such that $H_*(X;\Q)$ does not agree with $H_*(G;\Q)$ for any
compact Lie group $G$, as graded vector spaces.

This example is minimal, in the sense that any connected, finite loop space of rank less than $66$ is rationally equivalent to some compact Lie group $G$.
\end{thm}

In \cite{ABGP04} there is a list of which $p$--compact group to choose at each
prime. By the preceeding discussion, the problem of finding such a space is a
combinatorial problem, and one can
show that in high enough rank there
will be many examples.

\section{Steenrod's problem of realizing polynomial rings}\label{section:steenrod}
The 1960 ``Steenrod problem''
\cite{steenrod61,steenrod71}, asks, for a given ring $R$, which graded polynomial
algebras are realized as $H^*(Y;R)$ of some space $Y$, i.e., in which
degrees can the generators occur? In this section we give some background on this classical problem and
describe its solution in
\cite{AG08steenrod,AG09}.

Steenrod, in his original paper \cite{steenrod61}, addressed the
case of polynomial rings in a single variable: For $R =
\Z$ the only polynomial rings that occur are $H^*(\CP^\infty;\Z) \cong
\Z[x_2]$ and $H^*(\HP^\infty;\Z) \cong \Z[x_4]$, as he showed by
a short argument using his cohomology operations. Similarly, for $R=\F_p$ he showed that
the generator has to sit in degree $1$,$2$, or $4$ for $p=2$ and in degree $2n$
with $n | p-1$ for $p$ odd, but now as a consequence of Hopf invariant
one and its odd primary version (though it was not known at the time whether
the $p$ odd cases were realized when $n \neq 1,2$).

There were attempts to use the above
techniques to settle polynomial rings in several variable, but they gave only
very partial results. In the 70s, however, Sullivan's method, as generalized
by Clark--Ewing, realized many polynomial rings, as explained in
Section~\ref{section:existence}. Conversely, in the 80s Adams--Wilkerson
\cite{AW80short} and others
put restrictions on the potential degrees, using categorical properties of the category
of unstable algebras over the Steenrod algebra. This eventually led to the
result of Dwyer--Miller--Wilkerson \cite{DMW92} that
for $p$ large enough the Clark--Ewing examples are exactly the possible
polynomial cohomology rings over $\F_p$.

In order to tackle all primes, it turns out to be useful to have a
space-level theory, and that is what
$p$--compact groups provide. Namely, if $Y$ is a space such that $H^*(Y;\F_p)$
is a polynomial algebra, then the Eilenberg--Moore spectral sequence shows
that $H^*(\Omega Y;\F_p)$ is finite, and hence $Y\pcom$ is a $p$--compact
group.

\begin{thm}[Steenrod's problem, $\ch(R) \neq 2$ \cite{AG08steenrod}] Let $R$ be a commutative Noetherian
  ring of finite Krull dimension and let $P^*$ be a graded polynomial $R$--algebra in
  finitely many variables, all in positive even degrees. 

Then there exists a
  space $Y$ such that $P^* \cong H^*(Y;R)$ as graded algebras if and only if for
  each prime $p$ not a unit in $R$, the degrees of $P^*$ halved is a multiset
  union of the degrees lists occurring
  in Table~1 at that prime $p$, and the degree one, with the following exclusions (due to
  torsion): $(G(2,2,n), p =2)$ $n \geq 4$, $(G(6,6,2), p=2)$, $(G_{24}, p=2)$, $(G_{28}, p=2,3)$, $(G_{35},
  p=2,3)$, $(G_{36}, p=2,3)$, and $(G_{37}, p=2,3,5)$.
\end{thm}

When $\ch(R) \neq 2$, all generators are in even
degrees by anti-commutativity, so the assumptions of the theorem are satisfied.
The proof in \cite{AG08steenrod} only relies on the general
theory of $p$--compact groups, not on the classification. The case $R=\F_p$,
$p$ odd, was solved earlier by Notbohm \cite{notbohm99}, 
also using $p$--compact group theory.
Taking $R = \Z$ gives the old
conjecture that if $H^*(Y;\Z)$ is a polynomial ring, then it is isomorphic to
a tensor product of copies of $\Z[x_2]$, $\Z[x_4, x_6, \ldots, x_{2n}]$, and $\Z[x_4, x_8,
\ldots, x_{4n}]$, the cohomology rings of $\CP^\infty$,
$B\!\SU(n)$ and $B\!\Sp(n)$.

\begin{thm}[Steenrod's problem, $\ch(R) = 2$ \cite{AG09}] \label{Steenrod}
Suppose that $P^*$ is a gra\-ded polynomial algebra in finitely many variables
over a commutative ring $R$ of characteristic $2$. Then $P^* \cong H^*(Y;R)$ for a
space $Y$ if and only if
$$
P^* \cong H^*(BG;R) \otimes H^*(B\!\DI(4);R)^{\otimes r}\otimes
H^*(\RP^\infty;R)^{\otimes s} \otimes H^*(\CP^\infty;R)^{\otimes t}
$$
as a graded algebra, for some $r, s, t \geq 0$, where $G$ is a compact connected Lie group
with finite center.
In particular, if all generators of $P^*$ are in degree $\geq 3$ then
$P^*$ is a tensor product of the cohomology rings of
the classifying spaces of $\SU(n)$, $\Sp(n)$, $\Spin(7)$,
$\Spin(8)$, $\Spin(9)$, $G_2$, $F_4$, and $\DI(4)$.
\end{thm}

The proof reduces to $R = \F_2$, and then uses the classification of
$2$--compact groups.
It would be interesting to try to list all polynomial
rings which occur as $H^*(BG;\F_2)$ for $G$ a compact connected Lie
group with finite center.

One can also determine to which extend the space is unique. The
following result was proved by Notbohm \cite{notbohm99} for $p$
odd and \cite{AG09} for $p=2$, as the
culmination of a long series of partial results, started by Dwyer--Miller--Wilkerson
\cite{DMW86,DMW92}.

\begin{thm}[Uniqueness of spaces with polynomial $\F_p$--cohomology] \label{cohomologicaluniqueness} 
Suppose $A^*$ is a finitely generated polynomial
$\F_p$--algebra over the Steenrod algebra ${\mathcal A}_p$, with
generators in degree $\geq 3$. Then there exists, up to $p$--completion, at most one
homotopy type $Y$ with $H^*(Y;\F_p) \cong A^*$, as
graded algebras over the Steenrod algebra.

If $P^*$ is a finitely generated polynomial
$\F_p$--algebra, then there exists at most finitely many homotopy types $Y$,
up to $p$--completion, such that $H^*(Y;\F_p) \cong P^*$ as graded
$\F_p$--algebras.
\end{thm}

The assumption $\geq 3$ above cannot be
dropped, as easy examples show, and integrally uniqueness
rarely hold, as discussed in Section~\ref{section:loop}; see also
\cite{AG08steenrod,AG09}.

\enlargethispage*{10pt}
\section{Homotopical finite groups, group actions,\ldots}\label{section:vistas}
This survey is rapidly coming to an end, but we nevertheless want to
briefly mention some 
other recent developments in homotopical group theory.

In connection with the determination of the algebraic K-theory
of finite fields, Quillen and Friedlander proved the following: If $G$
is a reductive group scheme, and $q$ is a prime power, $p \nmid q$, then
$$BG(\F_q )\pcom \simeq (BG(\bbC)\pcom)^{h\langle \psi^q\rangle}$$
where the superscript means taking homotopy fixed-points of the self-map
$\psi^q$ corresponding to multiplication by $q$ on the root datum---it says that, at $p$, fixed-points and homotopy
fixed-points of the Frobenius map raising to the $q$th power agree.

The right-hand side of the equation makes sense with $BG(\bbC)\pcom$
replaced by a $p$--compact group.
Benson speculated in the mid 90s that the resulting object should be the
classifying 
space of a ``$p$--local finite group'', and be determined by a conjugacy
or fusion pattern on a finite $p$--group $S$, as axiomatized by Puig
\cite{puig06} (motivated by block theory),
together with a certain rigidifying $2$--cocycle. He even gave a candidate fusion pattern
 corresponding to $\DI(4)$, namely a fusion pattern
constructed by Solomon years earlier in connection with the classification of
finite simple groups, but shown 
not to exist inside any finite group \cite{benson98sporadic}.

All this turns
out to be true and more! 
A theory of $p$--local finite groups was
founded and developed
by Broto--Levi--Oliver in \cite{BLO03jams}, and has seen  rapid
development by both homotopy theorists and group theorists since
then. 
The Solomon $2$--local
finite groups $\Sol(q)$ were shown to exist in \cite{LO02}, and
a study of Chevalley $p$--local finite groups, $p$ odd, was initiated in \cite{BM07}. 
A number of exotic $p$--local finite groups have been found for $p$
odd, but the family $\Sol(q)$ remains
the only known examples at $p=2$, prompting the speculation that
perhaps they are the only exotic simple
 $2$--local finite groups!
Even partial results in this direction could have
implications for the proof of the classification of finite simple groups.
A modest starting point is the result in \cite{bcglo1} that any so-called {\em constrained} fusion pattern comes from a
(unique constrained) finite group---the result is purely group theoretic, and,
while not terribly difficult, the only known proof uses
techniques of a kind hitherto foreign to the
classification of finite simple groups.

One can ask for a theory more general than $p$--local finite groups, 
broad enough to encompass both $p$--completions of
arbitrary compact Lie groups and $p$--compact groups, and one such theory
was indeed developed in \cite{BLO07}, the so-called $p$--local compact
groups. One would like to identify connected $p$--compact groups
inside $p$--local compact groups in some group theoretic manner. This relates
to the question of describing the relationship between the classical Lie
theoretic structure and the $p$--fusion structure, mentioned several times before in this
paper; the
proof of the classification of $p$--compact groups may offer some hints on how to
proceed.

In a related direction, one may attempt to relax the
condition of compactness in $p$--compact groups to include more
general types of groups; the paper
\cite{CCS10} shows that replacing cohomologically finite by
noetherian gives few new examples.
An important class of groups to understand are Kac--Moody
groups, and the paper \cite{BK02} shows, amongst other things, that
homomorphisms from
finite $p$--groups to Kac--Moody groups still correspond to
maps between classifying spaces.
This gives hope that some of the
homotopical theory of maximal
tori, Weyl groups, etc.\ may also be brought to work in this setting,
but the correct
general definition of a homotopy Kac--Moody group is still
elusive, the Lie theoretic definition being via
generators-and-relations rather than intrinsic. A good understanding of
the restricted case of affine Kac--Moody groups and loop groups would already be very interesting.

Groups were historically born to act, a group action being a
homomorphisms from $G$ to the group of homeomorphisms of a space
$X$. In homotopy
theory, one is however often only given $X$ up to an equivariant map which is
a homotopy
equivalence. Here the
appropriate notion of an action is an element in the mapping
space $\map(BG,B\!\Aut(X))$, where as before $\Aut(X)$ denotes the space of self-homotopy
equivalences (itself an interesting group!).
Homotopical group actions can also be studied one prime at a time, and
assembled to global results afterwards. Of particular interest is the
case where $X$ is a sphere. Spheres are non-equivariantly determined by their dimension, and
self-maps by their degree. It turns out that something
similar is true for homotopical group actions of finite groups on
$p$--complete spheres \cite{GS}. But, one has to interpret dimension as meaning dimension function,
assigning to each $p$--subgroup of $G$ the homological dimension of the
corresponding homotopy fixed-point set, and correspondingly the degree
is a degree function,  viewed as an element in a certain $p$--adic Burnside ring. 
Furthermore there is hope to determine exactly which dimension
functions are realizable. 
Understanding groups is  homotopically open-ended\ldots

\begin{spacing}{0.9}
\bibliographystyle{plain}

\end{spacing}

\end{document}